\documentclass[11pt]{article}

\usepackage{amsmath,amsthm,amssymb, amsfonts,graphicx,mathrsfs,commath,bm}
\usepackage{dsfont,color,tikz,tkz-graph,enumerate,placeins}
\usepackage{cleveref}
\usetikzlibrary{arrows.meta}
\crefname{section}{§}{§§}
\Crefname{section}{§}{§§}
\newtheorem{theorem}{Theorem}[section]

\newtheorem{proposition}[theorem]{Proposition}

\newtheorem{assumption}[theorem]{Assumption}
\newtheorem{definition}[theorem]{Definition}

\newtheorem{remark}[theorem]{Remark}

\DeclareMathOperator{\Tr}{Tr}
\makeatletter

\newcommand{\Rmnum}[1]{\expandafter\@slowromancap\romannumeral #1@}

\makeatother

\graphicspath{{./Figures/}}

\begin{document}
\title{Mixed mode oscillations and phase locking \\ in coupled FitzHugh-Nagumo model neurons}
\author{Elizabeth N. Davison
\thanks{Department of Mechanical and Aerospace Engineering, Princeton University, Princeton, NJ 08540, USA.}
\and Zahra Aminzare
\thanks{ Program in Applied and Computational Mathematics, Princeton University, Princeton, NJ 08544, USA.}
\and Biswadip Dey$^{*}$
\and Naomi Ehrich Leonard $^{*}$
}
\date{}
\maketitle

\begin{abstract}
We study the dynamics of a low-dimensional system of coupled model neurons as a step towards understanding the vastly complex network of neurons in the brain. We analyze the bifurcation structure of a system of two model neurons with unidirectional coupling as a function of two physiologically relevant parameters: the external current input only to the first neuron and the strength of the coupling from the first to the second neuron. Leveraging a timescale separation, we prove necessary conditions for multiple timescale phenomena observed in the coupled system, including canard solutions and mixed mode oscillations. For a larger network of model neurons, we present a sufficient condition for phase locking when external inputs are heterogeneous. Finally, we generalize our results to directed  trees of model neurons with heterogeneous inputs.
\end{abstract}

\maketitle

\section{Introduction}
\label{intro}
%
The study of model neurons has a rich history, dating back to the pioneering work of Hodgkin and Huxley \cite{hodgkin1952quantitative} on the action potential in the squid giant axon. A two-dimensional model that captures salient qualities of the four-dimensional Hodgkin-Huxley model was developed independently by FitzHugh \cite{fitzhugh_mathematical_1955,fitzhugh1961impulses} and by Nagumo\cite{nagumo1962active}. In this model, commonly known as the FitzHugh-Nagumo (FN) model, one variable represents the membrane potential and the other represents a gating variable. A constant external input to the FN model neuron can produce quiescent behavior (a low-voltage stable equilibrium point), firing (a stable limit cycle), or saturated behavior (a high-voltage stable equilibrium point). The FN model neuron captures realistic neuronal behavior such as spike accommodation, bistability, and excitability \cite{izhikevich2007dynamical}. 

A system of two coupled neurons can represent a larger network of neurons that cluster into two groups in which neurons within each group synchronize but neurons in different groups do not.
 A cluster synchronized network can be reduced to a \emph{quotient network} \cite{russo_global_2010, schaub2016graph} by leveraging balanced conditions on coupling and graph structure \cite{sorrentino2016complete}, as well as  
bounds on coupling strength 
\cite{davison_sync_2016, aminzare2018cluster}.
A system of two FN model neurons with gap junction diffusive coupling (two-FN system) has been studied numerically and analytically in the symmetric case\cite{hoff_numerical_2014,campbell2001multistability}, where both neurons receive the same external input and are coupled bidirectionally (undirected coupling). Gap junction diffusive coupling is modeled as a difference between the membrane potentials of the two model neurons multiplied by a parameter that represents the coupling strength. The two-FN system has also been studied numerically\cite{hoff_numerical_2014} in a context where the intrinsic properties of both models are the same but the neurons are coupled unidirectionally (directed coupling). Here, we add to the existing literature by analytically describing the bifurcation structure of the directed two-FN system in terms of two parameters, the external input to the first model neuron and the unidirectional coupling strength from the first model neuron to the second.

The FN model neuron is a classic example of a fast-slow system, and the coupled pair of FN model neurons exhibits rich dynamics characterized by the timescale separation. Under certain conditions on external input and coupling strength, the system exhibits \emph{canard solutions}, which are solutions that pass from a stable to an unstable manifold in the slow system and stay near the unstable manifold for a long time relative to the slow system timescale \cite{benoit1990canards,szmolyan2001canards, kuehn2015multiple}. Canard solutions result from the presence of folded singularities, particularly folded nodes, which give rise to robust families of canard solutions \cite{szmolyan2001canards,wechselberger_existence_2005}. Canard solutions can lead to mixed mode oscillations (MMOs), which are periodic oscillations that alternate between canard-driven oscillations and a relaxation oscillation \cite{milik1998geometry}. The existence of canards and MMOs has been described for systems in four dimensions 
\cite{benoit1983systemes,szmolyan2001canards}, systems with two slow variables and two fast variables \cite{tchizawa_two_2013}, and generalized systems in arbitrary finite dimensions \cite{wechselberger2012propos}. The folded saddle node of type I (FSN I) and folded saddle node of type II (FSN II) have been identified as mechanisms for the onset of MMOs in fast-slow systems \cite{milik1998geometry, desroches_geometry_2008, krupa2010local, Curtu_Rubin_Canard-MMO, desroches2012mixed}. We leverage these results to determine the parameter regimes where canards and MMOs may be present in the directed two-FN system, which has two slow variables and two fast variables.	
	
Canard-induced MMOs have been studied analytically in numerous systems including chemical reactions \cite{petrov1992mixed,moehlis2002canards}, the Hodgkin-Huxley neuronal model \cite{rubin2007giant}, cortical grid cells \cite{rotstein2008canard}, and a self-coupled FN model neuron \cite{desrochesmixedmode2008}. In a two-FN system, the onset of firing, as coupling strength is increased, can be characterized by the appearance of canard solutions, and by MMOs, as the coupling is increased further. The existence of canard solutions in a two-FN system was first proven using nonstandard analysis in the case of model neurons with identical parameters \cite{tchizawa2002winding}. Necessary conditions were found in terms of a model parameter that controls the slope of the linear nullcline of the system. Conditions for different stability types of folded singularities were found in terms of the same model parameter in a slightly modified, but still symmetric, model \cite{ginoux2015canards}. Here, we fix the corresponding parameter within the range where canard solutions may be present and find conditions for existence of canard-induced MMOs in terms of two parameters that break symmetry: external input and coupling strength.

A condition for the onset of MMOs in a two-FN system was shown as an application of a method developed to study MMOs in systems with two fast variables and two slow variables \cite{ krupa2014weakly}. There are no symmetry requirements and the main result is a necessary condition for MMO onset in terms of a parameter corresponding to the input to one of the neurons. In the spirit of this work, we prove explicit necessary conditions for existence of canard solutions and MMOs in the directed two-FN system in terms of both the external input and the coupling strength. First, we take the singular limit of the system, and obtain necessary conditions on the bifurcation parameters for existence of transcritical bifurcations. The transcritical bifurcations in the singularly perturbed system delineate regions in parameter space where MMOs exist in the original system.

We show, further, that the original system admits Hopf bifurcations within a distance of order $\epsilon$ around the point in the parameter space where the singularly perturbed system admits transcritical bifurcations. 
This we use to derive novel bounds for phase-locking in representative networks of model neurons.
Phase locking is a generalization of synchronization where the phases of oscillating models remain separated by a constant offset, while amplitudes and waveforms may vary \cite{rosenblum1996phase}. A common phenomenon in nature, phase locking has been studied in cardiac rhythms \cite{van1928lxxii, glass1982fine,guevara1981phase}, in the firing patterns of squid axons \cite{matsumoto1987chaos}, in two coupled phase oscillators \cite{kopell2002mechanisms}, in local field potential measurements of neurons in the human brain \cite{rutishauser2010human, wang2010neurophysiological}, and in the brain as a mechanism for coordination between groups of neurons \cite{varela2001brainweb}.

Finally, we consider the more general problem of $n$ FN model neurons linked by unidirectional gap junction diffusive coupling in a directed tree, with heterogeneous coupling strengths and heterogeneous external inputs. As in the directed two-FN system, this can represent a class of large networks that contain cluster synchronized groups of model neurons and satisfy conditions on graph structure \cite{sorrentino2016complete} and connectivity\cite{davison_sync_2016, aminzare2018cluster} so they can be reduced to a quotient network \cite{russo_global_2010, schaub2016graph}.  An analogous problem with homogeneous coupling strength has been analyzed in detail in the strong coupling limit where the dynamics are reduced using singular perturbation theory\cite{medvedev2001synchronization}.   Here, we leverage an analysis of the singular perturbation of the directed two-FN system to provide necessary conditions for the existence of MMOs and sufficient conditions for phase locking in the original $n$-FN system.

Our contributions towards understanding the dynamics of networked nonlinear model neurons are as follows. First, we explain how the bifurcation structure of the directed two-FN system relates to the bifurcation structure of the reduced, singularly perturbed system that is used to study canard solutions. This is critical because the reduced system can be used to explain features of the original system and the original system can be used to understand the reduced system. Second, we provide necessary conditions for canards and MMOs in the directed two-FN system in terms of two model parameters; this is an extension of the conditions found in terms of one parameter in the literature. Third, we provide a sufficient condition for phase locking given heterogeneous external inputs in the directed two-FN system. We generalize these conditions to directed trees of FN model neurons.

The paper is organized as follows. In Section \ref{oneFN}, we review the standard analysis of a single FN model neuron
 and give a biophysical rationale for bounds on model parameters used throughout the paper. In Section \ref{twoFN}, we define the directed two-FN system and find conditions for Hopf bifurcations. In Section \ref{canardsetc}, we compute the singular perturbation of the directed two-FN system. We prove necessary conditions for the transcritical bifurcations in the singularly perturbed system and canards and MMOs in the original system in Section \ref{mainres}. In Section \ref{chainFN}, we generalize the results to directed trees of FN model neurons. We provide a numerical example to illustrate our results.
%
%
%
%
%
\section{Single FitzHugh-Nagumo model neuron}
\label{oneFN}
%
The FN model is a two-dimensional reduction of the four-dimensional Hodgkin-Huxley (HH) model for the neuronal action potential. By letting $y$ and $z$ represent the membrane potential and a slow gating variable, respectively, its dynamics are given by
\begin{equation}
\begin{aligned}
\dfrac{dy}{dt} &= \psi(y) -z + I,
\\
\dfrac{dz}{dt} &= \epsilon (y-b z),
\end{aligned}
\label{FN1}
\end{equation}
where $\psi(y)$ is a cubic polynomial. For our purposes, we use $\psi(y) = y-\frac{y^3}{3}-a$. In this model, $I$ corresponds to an external input, $0<\epsilon\ll1$ is a positive timescale separation constant, and $a$ and $b$ are positive constants. The FN model is far simpler to analyze than the full HH model due to the lower dimension.

Despite the lower dimension, the FN model captures a range of physiologically meaningful regimes and behaviors \cite{izhikevich2007dynamical, rocsoreanu2012fitzhugh}. The bifurcation structure of a single FN model is complex, with a Hopf bifurcation and \emph{canard explosion}, or transition from small, Hopf-induced limit cycles to large relaxation oscillations\cite{guckenheimer2013nonlinear}, marking both the transition from quiescent to firing and the transition from firing to saturated. As $I$ is increased, the single stable equilibrium point transitions to relaxation oscillations, or stable limit cycles. As $I$ is further increased, the large stable limit cycles persist and small unstable limit cycles appear. The unstable limit cycles disappear in a subcritical Hopf bifurcation at $I_0$ and the equilibrium point becomes unstable. Figure \ref{FN_bifurcation_diagram} depicts the bifurcation diagram of the FN model when $I$ is varied. 
\begin{figure}[h]
	\centering
	\includegraphics[width =0.5\textwidth]{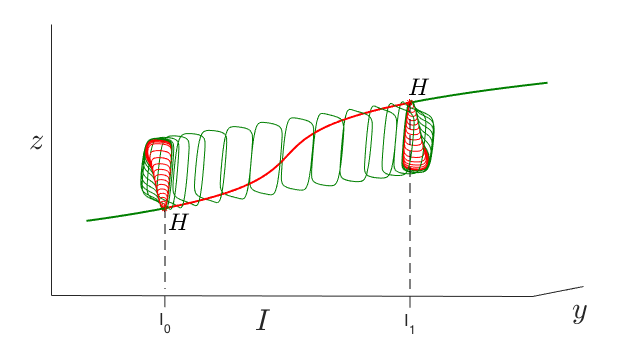}
	\caption{ Bifurcation diagram for a single FN model drawn with a numerical continuation software tool \cite{dhooge2003matcont} for $a = 0.875$, $b =0.8$, and $\epsilon= 0.08$. Green corresponds to stable equilibrium points or limit cycles and red corresponds to unstable equilibrium points or limit cycles.
    When $I$ is small, $I<I_0$, the FN model is in the quiescent regime. For $I<I_0$ where $I$ is sufficiently close to $I_0$, the system concurrently exhibits a stable fixed point, small unstable oscillations, and large stable relaxation oscillations. The FN model is in the firing regime when $I_0<I<I_1$. For $I>I_1$ where $I$ is sufficiently close to $I_1$, the system concurrently exhibits a stable fixed point, small unstable oscillations, and large stable relaxation oscillations.  Finally, for $I>I_1$, the FN model is in the saturated regime.}
	\label{FN_bifurcation_diagram}
\end{figure}

The FN model is a suitable choice for network analysis because it is both dynamically rich and analytically tractable.  A necessary condition\cite{Rinzel1981} for the FN model to exhibit distinct quiescent, firing, and saturated regimes is the existence of a unique equilibrium point for all values of the bifurcation parameter $I$. In this paper we assume the following:
\begin{assumption}
Parameters $a$, $b$, and $\epsilon$ are such that the FN model (\ref{FN1}) has a unique equilibrium point for all values of $I\geq 0$. This results in conditions $0<a<1$ and $0<b<1$. \label{assumption}
\end{assumption}
The condition on $a$ corresponds to a simple voltage offset requirement and the condition on $b$ corresponds to a requirement that the slope of the linear nullcline of (\ref{FN1}) must be greater than that of the cubic nullcline of (\ref{FN1}).

The following proposition from Ref.~\cite{guckenheimer2013nonlinear}  
describes the stability of the unique equilibrium point of (\ref{FN1}) given Assumption \ref{assumption} and conditions on $I$ for Hopf bifurcations.

\begin{proposition} (Ref.~\cite{guckenheimer2013nonlinear})
Let Assumption \ref{assumption} hold. Then, there exists $I_0<I_1$ such that the equilibrium point is stable for $I<I_0$ and, as $I$ increases, it will undergo a transition to an unstable equilibrium point through a Hopf bifurcation at $I_0$. As $I$ is increased further it will undergo a transition from unstable to stable through a second Hopf bifurcation at $I_1$. 
\end{proposition}

We review the approach to analyzing stability of the limit cycles arising from the Hopf bifurcations following  the methods of Chapter 3 of Ref.~\cite{guckenheimer2013nonlinear} and the application to the FN model in Ref.~\cite{holmes_notes}.  We  generalize the approach to networks of FN model neurons in later sections.

The dynamics at a Hopf bifurcation at the origin of a two-dimensional system can be written as
\begin{align}
\left(\begin{array}{c} \dfrac{dx_1}{dt}\\ \dfrac{dx_2}{dt}\end{array}\right) = \left(\begin{array}{cc} 0 & -\omega\\ \omega & 0 \end{array}\right)\left(\begin{array}{c} x_1\\ x_2\end{array}\right)+\left(\begin{array}{c} F(x_1, x_2)\\ G(x_1,x_2) \end{array}\right),\label{form_Hopf}
\end{align}
such that $F$ and $G$ satisfy $F(0,0) = G(0,0) = 0$ and $D_\mathbf{x} F(0,0) = D_\mathbf{x}G(0,0)=\mathbf{0}$ where $\mathbf{x} = [x_1, \; x_2]$.

\begin{definition}
[Cubic coefficient \cite{hassard1978bifurcation, guckenheimer2013nonlinear}] Consider the system $\dot{\mathbf{x}} = \mathbf{f}(\mathbf{x},\gamma)$.
The coefficient of the cubic term of the Taylor expansion of $\mathbf{f}$, (\ref{form_Hopf}) is expressed as
\begin{align}
	\alpha &= \frac{1}{16}(F_{x_1 x_1 x_1}+F_{x_1 x_2 x_2}+G_{x_1 x_1 x_2}+G_{x_2 x_2 x_2})\biggr\rvert_{(0,0)}\notag\\
	&+\frac{1}{16\omega}(F_{x_1 x_2}(F_{x_1 x_1}+F_{x_2 x_2})
    -G_{x_1 x_2}(G_{x_1 x_1}+G_{x_2 x_2})\notag\\
    &-F_{x_1 x_1}G_{x_1 x_1}+F_{x_2 x_2}G_{x_2 x_2})\biggr\rvert_{(0,0)},
\label{CCdefn}
\end{align} 
where $F_{x_1 x_2}$ denotes $\dfrac{\partial^2 F}{\partial x_1 \partial x_2}$, and so on. \label{cubiccoeff}
\end{definition}
\begin{proposition}[Theorem 3.4.2 (modified) \cite{guckenheimer2013nonlinear}]
The system $\dot{\mathbf{x}} = \mathbf{f}(\mathbf{x},\gamma)$, admits a Hopf bifurcation for the parameter value $\delta_0$ at an equilibrium point $\mathbf{p}$ if
\begin{enumerate}
\setlength{\itemsep}{0pt}
\setlength{\parskip}{0pt}
\item $D_\mathbf{x} \mathbf{f}(\mathbf{p},\delta_0)$ has a pair of pure imaginary eigenvalues and no other eigenvalues with zero real parts.
\item $ \dfrac{\partial}{ \partial {\delta}}\Re(\lambda(\delta))\biggr \rvert_{\delta = \delta_0} \neq 0$, where $\Re(\lambda)$ denotes the real part of the eigenvalue $\lambda$.
\item The cubic coefficient of the Taylor expansion of $\mathbf{f}$, denoted by $\alpha$ and defined in Definition \ref{cubiccoeff}, is nonzero.
\end{enumerate}
Furthermore, if $\alpha<0$, the Hopf bifurcation is supercritical, while, if $\alpha>0$, the Hopf bifurcation is subcritical.\label{Hopfcond}
\end{proposition}

The cubic coefficient is also called the first Lyapunov coefficient. It can be used to determine the location of a solution that separates small unstable limit cycles from larger stable limit cycles and so delineates the canard explosion \cite{kuehn2015multiple}. The value of the bifurcation parameter $I$ at the canard explosion
 is close to the value of $I$ at the Hopf bifurcation, 
 as illustrated in Figure \ref{FN_bifurcation_diagram}. 

For the FN model, the cubic coefficient is given by
\begin{align*}
\alpha &=\dfrac{1}{8}\left(\dfrac{2b-b^2 \epsilon-1}{1-b^2\epsilon}\right).
\end{align*}
In this paper we choose parameters that ensure Assumption~\ref{assumption} holds and the bifurcations are subcritical Hopf ($\alpha>0$); these yield biologically realistic dynamics \cite{baer1986singular,rinzel1983hopf}. We fix $a=0.875$, $b=0.8$, and $\epsilon = 0.08$, and we consider $I\ge 0$ as a bifurcation parameter.
%
%
%

%
%
\section{Directed two-FN model neuron system}
\label{twoFN}

The directed two-FN model neuron system is shown in Figure \ref{2FNfig}. The first model neuron is denoted $A$ and it receives external input $I$. The second model neuron is denoted $B$ and it receives no external input.  The coupling is unidirectional from $A$ to $B$, with coupling strength $\gamma$. $A$ and $B$ have the same intrinsic dynamics, i.e., the same values of $a$, $b$, and $\epsilon$ as defined above. We let $I$ and $\gamma$ be bifurcation parameters. The directed two-FN system corresponds to the reduced cluster synchronized setting where there are many more model neurons in one cluster (the first) than the other (the second).
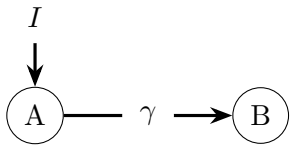
\begin{figure}[b!]
\begin{center}
	\begin{tikzpicture}
	\begin{scope}[>={Stealth[black]},
	every node/.style={fill=white,circle},
	every edge/.style={draw=black,very thick}]
	\node[shape=circle,draw=black] (A) at (0,0) {A};
	\node[shape=circle,draw=black] (B) at (3,0) {B};
	\path [->] (A) edge node {$\gamma$} (B);
	\path [->] (0,1.5) edge node [above]{$I$} (A) ;
	\end{scope}
	\end{tikzpicture}
\end{center}
\caption{A directed network of two FN model neurons, $A$ and $B$.  $A$ receives an external input $I$ and there is a unidirectional coupling from $A$ to $B$ with strength $\gamma$.}\label{2FNfig}
\end{figure}

The equations for the directed two-FN system are
\begin{subequations}\label{FN2}
\begin{align}
	\dfrac{d{y}_A}{dt} &= \zeta_A(y_{A},z_{A}, y_{B}, z_{B})= y_A-\frac{y_A^3}{3}-a -z_A+I,\label{FN2a}\\
	\dfrac{d{z}_A}{dt} &= \epsilon\xi_A(y_{A},z_{A}, y_{B}, z_{B})= \epsilon (y_A-b z_A),\label{FN2b}\\
	\dfrac{d{y}_B}{dt} &=\zeta_B (y_{A},z_{A}, y_{B}, z_{B})= y_B-\frac{y_B^3}{3} -a -z_B+\gamma (y_A-y_B),\label{FN2c}\\
	\dfrac{d{z}_B}{dt} &= \epsilon\xi_B(y_{A},z_{A}, y_{B}, z_{B})= \epsilon (y_B-b z_B). \label{FN2d}
\end{align}
\end{subequations}
Here, $y_A$ ($y_B$) is the membrane potential of $A$ ($B$) and $z_A$ ($z_B$) represents a slow gating variable in $A$ ($B$).

The bifurcation structure of directed and undirected two-FN systems have been studied extensively from a numerical perspective \cite{hoff_numerical_2014}. The bifurcation structure of the undirected system has been studied through analytical methods that leverage symmetry-based arguments or assume symmetric or near-symmetric FN models \cite{campbell2001multistability, tchizawa2002winding, ginoux2015canards}. In contrast we examine the system with asymmetry in both the external input and the coupling.

We begin by classifying the behavior of the two model neurons $A$ and $B$ in the $I$-$\gamma$ parameter space, as shown in Figure \ref{Igamma_bifn}. Regions (1)-(7) are described as follows:
\begin{enumerate}[(i)]
\setlength{\itemsep}{0pt}
\setlength{\parskip}{0pt}
\item For $I$ less than the value where $A$ undergoes the lower Hopf bifurcation, $I<I_{0A}$, both $A$ and $B$ will be quiescent at a stable equilibrium point. This corresponds to region (1) of Figure \ref{Igamma_bifn}. See Section \ref{quiescence} for details.
\item For $I$ greater than the value where $A$ undergoes the upper Hopf bifurcation,  $I>I_{1A}$, $A$ becomes saturated and, as $\gamma$ varies, there are three distinct behaviors  for $B$. See Section \ref{abovebd} for  details. For $I>I_{1A}$, there exist $I_{0B}(I)$ and $I_{1B}(I)$ such that
\begin{enumerate}
\setlength{\itemsep}{0pt}
\setlength{\parskip}{0pt}
\item For $\gamma<I_{0B}(I)$, $B$ is quiescent. This corresponds to region (2) in Figure \ref{Igamma_bifn}.
\item For $I_{0B}(I)<\gamma<I_{1B}(I)$, $B$ is firing. This corresponds to region (3) in Figure \ref{Igamma_bifn}.
\item For $\gamma>I_{1B}(I)$, $B$ is saturated. This corresponds to region (4) in Figure \ref{Igamma_bifn}.
\end{enumerate}
\item For $I_{0A}<I<I_{1A}$, $A$ is firing and there are three distinct behaviors for $B$:
\begin{enumerate}
\setlength{\itemsep}{0pt}
\setlength{\parskip}{0pt}
\item When $\gamma>1-b\epsilon$, $B$ is phase locked with  $A$. This corresponds to region (5) of Figure \ref{Igamma_bifn}. See Section \ref{highg} for details.
\item When $\gamma<1-b\epsilon$ and $I$ is below a curve, denoted by $I_*(\gamma)$, MMOs or small canard oscillations may be present. This corresponds to region (6) of Figure \ref{Igamma_bifn}. See Section \ref{canardsetc} for the derivation and Section \ref{MMObegin} for details.
\item When $\gamma<1-b\epsilon$ and $I$ is above $I_*(\gamma)$, $B$ is firing. This corresponds to region (7) of Figure \ref{Igamma_bifn}. See Section \ref{canardsetc} for the derivation and Section \ref{MMObegin} for details.
\end{enumerate}
\end{enumerate}
\begin{figure}[t!]
	\centering
	\includegraphics[width = 0.5\textwidth]{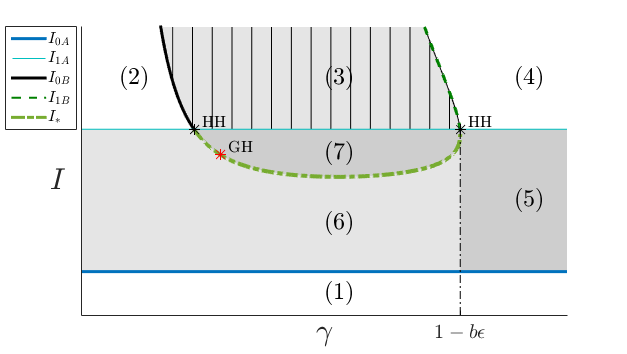}
	\caption{Regions of behavior of the directed two-FN system (\ref{FN2}) in the $I$-$\gamma$ parameter space. Boundaries between regions are identified in the key. In regions (3), (5), (6), and (7), shaded gray, there is a stable limit cycle such that either $A$ or $B$ is firing. In region (3), with cross hatching, only $B$ is firing. In regions (5) and (7), in darker gray, there is phase locking. In region (6), in light gray, $A$ is firing and $B$ may exhibit canard solutions. All boundaries are computed analytically. HH denotes a Hopf-Hopf bifurcation and GH denotes a generalized Hopf bifurcation.}
	\label{Igamma_bifn}
\end{figure}

In Section~\ref{canardsetc} we desingularize the system (\ref{FN2}) by performing a singular perturbation and changing timescales. We analyze the resulting second-order dynamics and draw conclusions about canards and MMOs for the original dynamics (\ref{FN2}) in Section \ref{mainres}.

\section{Fast-slow phenomena in the directed two-FN system}
\label{canardsetc}
In this section, we assume $A$ is firing and study the onset of firing in $B$ as $\gamma$ increases. We begin by providing definitions of canards and MMOs, which are observed numerically at the transition from quiescent to firing in $B$. For a general fast-slow system expressed as
\begin{equation}
\begin{aligned}
\dfrac{d\mathbf{y}}{dt} 
&= \bm{f}(\mathbf{y},\mathbf{z}),
\\
\dfrac{d\mathbf{z}}{dt} 
&= \epsilon \bm{g}(\mathbf{y},\mathbf{z}),
\end{aligned}
\label{fssys}
\end{equation}
$\mathbf{y}\in \mathds{R}^m$ are fast variables, $\mathbf{z}\in \mathds{R}^n$ are slow variables, and $0<\epsilon \ll 1$ is the timescale separation. The singular limit corresponds to $\epsilon = 0$.

\begin{definition}[Critical manifold]
	Given system (\ref{fssys}) with $\epsilon=0$, 
	\[C= \left\{ (\mathbf{y},\mathbf{z})\in \mathds{R}^m \times \mathds{R}^n : \bm{f}(\mathbf{y},\mathbf{z}) = \mathbf{0} \right\}\]
	 is called the {\em critical manifold}.
\end{definition}
\begin{definition}[Fold curve \cite{kuehn2015multiple}]
\label{foldcurve}
Given system (\ref{fssys}), denote the curve of points in $C$ that are not normally hyperbolic as $L = \left\{(\mathbf{y},\mathbf{z})\in C : D_\mathbf{y}\bm{f}(\mathbf{y},\mathbf{z}) = 0, D_{\mathbf{y}\mathbf{y}}\bm{f}(\mathbf{y},\mathbf{z}) \neq 0\right\}$. $L$ is called the \emph{fold curve}. The singularities on the fold curve are called {\em folded singularities}.
\end{definition}
\begin{definition}[Canard] 
A solution of (\ref{fssys}) is called a \emph{canard} if it stays within $\mathscr{O}(\epsilon)$ of a repelling branch of the critical manifold for a time that is $\mathscr{O}(1)$ on the slow timescale, $\tau_1=t\epsilon$.
\end{definition}
\begin{definition}[Mixed mode oscillation (MMO) \cite{kuehn2015multiple}]
Periodic solutions of (\ref{fssys}) with peaks of substantially different amplitudes are called {\em MMOs}. Canard solutions often comprise the small oscillations present in MMOs. 
\end{definition}

To find canards in system (\ref{fssys}), we first examine the dynamics of (\ref{fssys}) in the singular limit. A singular canard\cite{wechselberger2012propos} is a solution that crosses from the attracting to repelling branch of the critical manifold through a folded singularity. In the singular limit, a singular canard predicts\cite{wechselberger2012propos} a family of canard solutions in (\ref{fssys}). An example of a canard solution and MMOs found in the directed two-FN system is shown in Figure \ref{canardMMO}.

\begin{figure}[hb!]
	\centering
	\includegraphics[width =0.25\textwidth]{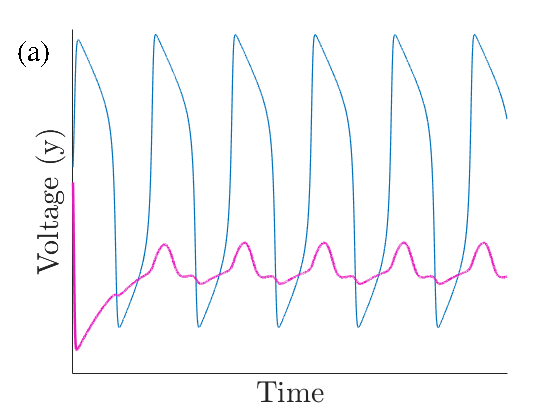}\quad
	\includegraphics[width =0.25\textwidth]{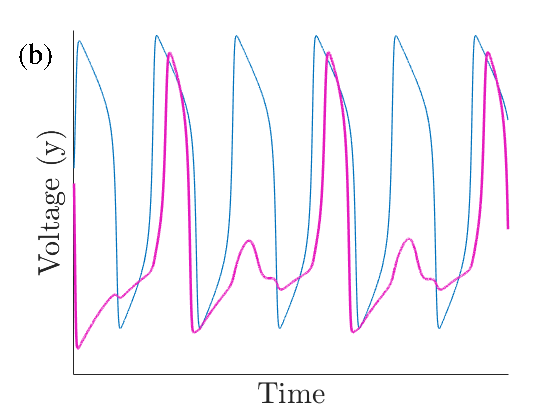}
	\caption{Example of canard solutions and MMOs observed in simulation of the directed two-FN system. In both plots $y_{A}$ is in blue and $y_{B}$ is in magenta. (a) For $I = 1$ and $\gamma=0.05$, $y_{B}$ follows a canard solution exhibited as small oscillations with the same frequency as the firing of model neuron $A$. (b) For $I = 1$ and $\gamma=0.08$, $y_{B}$ exhibits MMOs where the small oscillations are a canard solution.}
	\label{canardMMO}
\end{figure}

The remainder of this section details the calculations involved in the transformation of (\ref{FN2}) into a two-dimensional desingularized system, bifurcation and phase plane results for the desingularized system, and computation of the types of bifurcations. We implement methods that have been developed to determine the presence of canards and MMOs in general systems with $m$ fast variables and $n$ slow variables \cite{desroches2012mixed} and methods that have been applied to similar problems in two-FN systems \cite{freire2011stern,krupa2014weakly,ginoux2015canards}. In Section \ref{mainres}, we prove that solutions are canards during the onset of firing for sufficiently large coupling strength.

\subsection{The critical manifold and fold curve of the directed two-FN system}

In the slow timescale $\tau_1=t\epsilon$ system \eqref{FN2} becomes
\begin{equation}
\begin{aligned}
\epsilon\;\dfrac{d{y}_A}{d\tau_1} &= y_A-\frac{y_A^3}{3}-a -z_A +I,
\\
\epsilon\;\dfrac{d{y}_B}{d\tau_1} &= y_B-\frac{y_B^3}{3} -a- z_B+\gamma (y_A-y_B),
\\	
\dfrac{d{z}_A}{d\tau_1} &= y_A-b z_A,
\\
\dfrac{d{z}_B}{d\tau_1} &= y_B-b z_B. 
\end{aligned}
\label{eq1_slow_time_scale}
\end{equation}
The critical manifold, $C$, can be found by setting $\epsilon = 0$ in the fast two equations:
\begin{align}
C &= \left\{y_A, y_B, z_A, z_B
\left|\begin{aligned}
\zeta_A(y_A, y_B, z_A, z_B) = 0,\;  \zeta_B(y_A, y_B, z_A, z_B) = 0
\end{aligned}\right. \right\},
\end{align}	
where $\zeta_A$ and $\zeta_B$ are defined in \eqref{FN2}.

We denote by $S_{0}$ a compact submanifold of the smooth manifold $C$. $S_0$ is ``normally hyperbolic" if the eigenvalues of the Jacobian evaluated at each point $p \in S_0$ are nonzero. We can select a compact submanifold of $C$ that does not contain the singular points and so will be normally hyperbolic  \cite{kuehn2015multiple}. 

By Fenichel's Theorem, the slow dynamics of the two-FN system (\ref{FN2}) will lie $\mathscr{O}(\epsilon)$ away from $S_0$, on a normally hyperbolic slow invariant manifold $S_\epsilon$ with the same stability properties as $S_0$. \cite{kuehn2015multiple, wechselberger_canard_2013}

The fold curve is the set of singular points of $C$ as defined in Definition \ref{foldcurve}, or the points where the determinant of the Jacobian with respect to $y_A$ and $y_B$ is zero. The fold curve for (\ref{eq1_slow_time_scale}) satisfies
\begin{displaymath}
\begin{aligned}
y_A-\frac{y_A^3}{3}-a -z_A +I &= 0,
\\
y_B-\frac{y_B^3}{3} -a- z_B+\gamma (y_A-y_B) &= 0,
\\
(1-y_A^2)(1-y_B^2-\gamma) &= 0.
\end{aligned}
\end{displaymath}

\subsection{The singular limit of the directed two-FN system}
In the singular limit of (\ref{eq1_slow_time_scale}) when $\epsilon = 0$, the reduced two-FN system is 
\begin{subequations}\label{sing_lim}
\begin{align}
0 &= y_A-\frac{y_A^3}{3}-a -z_A +I, \label{sing_lima}\\
0 &= y_B-\frac{y_B^3}{3} -a- z_B+\gamma (y_A-y_B), \label{sing_limb}\\	
\dfrac{d z_A}{d\tau_1}&= y_A-b z_A, \label{sing_limc}\\
\dfrac{d z_B}{d\tau_1}&= y_B-b z_B. \label{sing_limd}
\end{align}
\end{subequations}

Following Chapter 19 in Ref. \cite{kuehn2015multiple}, we differentiate \eqref{sing_lim} with respect to $\tau_1$ to get
\begin{equation}
(D_\mathbf{y} \bm{\zeta}) \cdot \dfrac{d \mathbf{y}}{d\tau_1}
+
(D_\mathbf{z} \bm{\zeta}) \cdot \dfrac{d \mathbf{z}}{d\tau_1}
=
\mathbf{0},
\label{Dff_11}
\end{equation}
where $\bm{\zeta} = \left(\begin{array}{c} \zeta_A \\ \zeta_B \end{array}\right)$, $\mathbf{y} = \left(\begin{array}{c} y_A \\ y_B \end{array}\right)$ and $\mathbf{z} = \left(\begin{array}{c} z_A \\ z_B \end{array}\right)$. Using (\ref{sing_limc})-(\ref{sing_limd}), we can rewrite \eqref{Dff_11} as
\begin{align}
\det{(D_\mathbf{y} \bm{\zeta})}\dfrac{d \mathbf{y}}{d\tau_1} 
&= 
- \text{adj}(D_\mathbf{y} \bm{\zeta}) D_\mathbf{z}\bm{\zeta} \cdot (\mathbf{y}-b\mathbf{z}),
\label{intermediate}
\end{align}
where
\[D_\mathbf{y} \bm{\zeta} = \left(\begin{array}{cc} 1-y_{A}^2 & 0 \\ \gamma & 1-y_{B}^2-\gamma\end{array}\right),\; 
D_\mathbf{z} \bm{\zeta} = - \mathbf{I}_{2\times2},\]
 and $\mathbf{I}_{2\times2}$ is the identity matrix. The determinant of $D_\mathbf{y} \bm{\zeta}$ is zero at $y_{A*} = \pm 1$ and $y_{B*}= \pm \sqrt{1-\gamma}$.
 
 The two-dimensional reduced system (\ref{intermediate}) can be examined, as follows, by rescaling time to obtain the desingularized system described in Ref. \cite{wechselberger_canard_2013}.

\subsection{Desingularization of the directed two-FN system}
\label{desing_section}

Rescaling time in \eqref{intermediate} by $\tau_1 = -\det{(D_\mathbf{y} \bm{\zeta})}\tau_2$ yields the desingularized system:
\begin{equation}
\begin{aligned}
\dfrac{d y_A}{d\tau_2} 
&= \rho_1(y_A, y_B) 
= (1-y_{B}^2-\gamma)(y_{A} -bz_{A} ),
\\
\dfrac{d y_B}{d\tau_2} 
&= \rho_2(y_A,y_B) 
= -\gamma(y_{A} -bz_{A} )+(1-y_{A}^2)(y_{B}-bz_{B}),
\end{aligned}
\label{desing}
\end{equation}
where $z_A$ and $z_B$ are defined as
\begin{equation}
\begin{aligned}
z_A&= y_{A}-\frac{y_{A} ^3}{3} -a+I,
\\
z_B &= y_{B}-\frac{y_{B}^3}{3}-a+\gamma (y_{A} -y_{B})\label{zero_p}.
\end{aligned}
\end{equation}

The singularities of \eqref{intermediate} are the points where $\text{adj} (D_\mathbf{y} \bm{\zeta})(D_\mathbf{z} \bm{\zeta}) \cdot (\mathbf{y}-b\mathbf{z})=\mathbf{0}$. There are two types of singularities \cite{wechselberger_canard_2013} that are present in \eqref{desing}. The first is denoted an \textit{ordinary singularity}, which arises when $\det(D_\mathbf{y}\bm{\zeta})\neq 0$.
The second type of singularity is the \textit{folded singularity}, which arises when $\det{(D_\mathbf{y} \bm{\zeta})}=0$.

The ordinary singularities of \eqref{desing} correspond to the equilibrium of \eqref{intermediate}, which is the equilibrium of the original system \eqref{FN2} in the $\mathbf{y}$ plane. In Figure \ref{ordinary}, the signs of the eigenvalues associated with the ordinary singularity are depicted as a function of $I$ and $\gamma$. These show stability changes in \eqref{desing} that correspond to Hopf bifurcations in \eqref{FN2}. The curves in $I$-$\gamma$ space that delineate the different signs of the real parts of the eigenvalues are slightly different from the Hopf bifurcation curves because predictions from singular perturbation analysis are accurate up to $\mathscr{O}(\epsilon)$. In regions where $A$ and $B$ are both quiescent or saturated (dark gray), both eigenvalues have negative real parts and the ordinary singularity is a stable equilibrium. In regions where both $A$ and $B$ are firing (light blue), the real parts of both eigenvalues are positive and the ordinary singularity is an unstable equilibrium. In regions where either $A$ or $B$ is firing (light gray), the ordinary singularity is a saddle.

\begin{figure}[h]
\begin{center}
	\includegraphics[width =0.5\textwidth ]{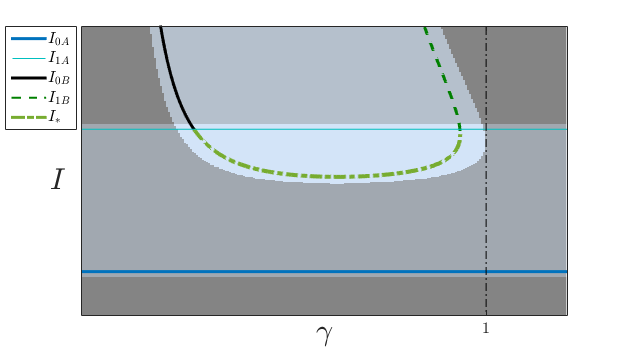}
	\caption{Regions in the $I$-$\gamma$ parameter space distinguishing local stability of the ordinary singularity in the desingularized system (\ref{desing}). Dark gray indicates two negative eigenvalues, light blue indicates two positive eigenvalues, and light gray indicates one positive and one negative eigenvalue. The Hopf bifurcations and distinguishing features of the original two-FN system \eqref{FN2}, the boundaries in Figure \ref{Igamma_bifn}, are plotted for comparison as five curves.} \label{ordinary}
\end{center}
\end{figure}

The folded singularities in \eqref{desing}  correspond to singularities on the fold curve in \eqref{FN2}, and they satisfy 
\begin{displaymath}
\det{(D_\mathbf{y} \bm{\zeta})} = 0,
\quad \textrm{and} \quad 
\text{adj} (D_\mathbf{y} \bm{\zeta})(D_\mathbf{z} \bm{\zeta}) \cdot (\mathbf{y}-b\mathbf{z}) = \mathbf{0},
\end{displaymath}
equivalently
\begin{subequations}
\label{foldedsing}
\begin{align}
(1-y_{B*}^2-\gamma)(1-y_{A*}^2) &= 0,
\label{foldedsinga}
\\
(1-y_{B*}^2-\gamma)(y_{A*} -bz_{A*} ) &= 0,
\label{foldedsingb}
\\  
(1-y_{A*}^2)(y_{B*}-bz_{B*}) -\gamma(y_{A*} -bz_{A*} ) &= 0.
\label{foldedsingc}
\end{align}
\end{subequations}
\eqref{foldedsinga} is satisfied when $y_{A*} = \pm 1$ or  $y_{B*} = \pm \sqrt{1-\gamma}$.

When $y_{A*} = \pm 1$, solving \eqref{foldedsingb}-\eqref{foldedsingc} results in $y_{A*} = bz_{A*}$, or $\gamma = 0$ and $y_{B*} = \pm 1$. Combining $y_{A*} = bz_{A*}$ with (\ref{zero_p}) results in $I = \pm \frac{1}{b}\mp\frac{2}{3}+a$, which is independent of $\gamma$ and so we do not consider this case further.

When $y_{B*} = \pm \sqrt{1-\gamma}$, we use \eqref{zero_p} to solve \eqref{foldedsingc} for $y_{A*}$, which is equivalent to solving the cubic equation
\begin{equation}
\beta_3 y_{A*}^3+\beta_2 y_{A*}^2 +\beta_1 y_{A*}+\beta_0 
= 0,
\label{cubic-4-yaStar}
\end{equation}
where
\begin{align*}
\beta_0 
&= 
b\gamma(I-a)+y_{B*}+b\left(-y_{B*}+\frac{y_{B*}^3}{3}+a+\gamma y_{B*}\right),
\\
\beta_1
&=
- b\gamma,
\\
\beta_2
&=
-y_{B*}+b\left(y_{B*}-\frac{y_{B*}^3}{3}-a-\gamma y_{B*}\right),
\\
\beta_3
&=
\frac{2b \gamma}{3}.
\end{align*}

The solutions of \eqref{cubic-4-yaStar} for $y_{A*}$ as a function of $\gamma$ and $I$ are:
\begin{align*}
y_{A*,k} &= -\dfrac{1}{3\beta_3}\left(\beta_2+C_k+\dfrac{\beta_2^2-3\beta_1 \beta_3}{C_k}\right),\end{align*}
where, for  $k = 1, 2, 3$,
\begin{align*}
C_k &= \left(\frac{\sqrt{-3}-1}{2}\right)^{k-1}\left(\dfrac{\sigma-\sqrt{-27\beta_3^2 \Delta}}{2}\right)^{1/3},\\
\Delta &= 18\beta_3\beta_2\beta_1\beta_0-4\beta_0\beta_2^3+\beta_2^2\beta_1^2-4\beta_3\beta_1^3-27\beta_3^2\beta_0^2,\\
\sigma&=2\beta_2^2 -9\beta_3 \beta_2 \beta_1+27 \beta_3^2 \beta_0.
\end{align*}

If $\Delta>0$, there are three real solutions and, if $\Delta<0$, there is one real solution.

The Jacobian of \eqref{desing} for the folded singularities with $y_{B*} = \pm \sqrt{1-\gamma}$ has the form:
\[D_\mathbf{y} \bm{\rho} (y_{A*}, y_{B*}) = \left(\begin{array}{cc}
0 & -2y_{B*}\xi_A\\
-\gamma-2y_{A*}\xi_B & 1-y_{A*}^2 \\
\end{array}\right),\]
where $\bm{\rho} = (\rho_1, \rho_2)^\top$, $\xi_A = y_{A*}-bz_{A*}$, and $\xi_B = y_{B*} -b z_{B*}$.
The trace and determinant of the Jacobian are used to classify each singularity. They are given by
\begin{align*}
\Tr{(D_\mathbf{y} \bm{\rho} (y_{A*}, y_{B*}))} 
&= 1 - y_{A*}^2,
\\
\det{(D_\mathbf{y} \bm{\rho} (y_{A*}, y_{B*}))} 
&= -2y_{B*}\xi_A\left(\gamma+2y_{A*}\xi_B\right).
\end{align*}
When $\det{(D_\mathbf{y} \bm{\rho} (y_{A*}, y_{B*}))} >0$, the eigenvalues have the same sign, so the singularity is a folded node. The stability of the node can be determined by looking at the sign of the trace. When $\det{(D_\mathbf{y} \bm{\rho} (y_{A*}, y_{B*}))}<0$, the eigenvalues have opposite signs, and the singularity is a folded saddle.

In Figure \ref{evals_folded}, the signs of the eigenvalues associated with the folded singularities are depicted as a function of $I$ and $\gamma$. The white regions in this figure correspond to values of $I$ and $\gamma$ where the desingularized system (\ref{desing}) has a single equilibrium point (i.e. $\Delta < 0$).

\begin{figure*}
\begin{center}
	\includegraphics[width=0.8\textwidth]{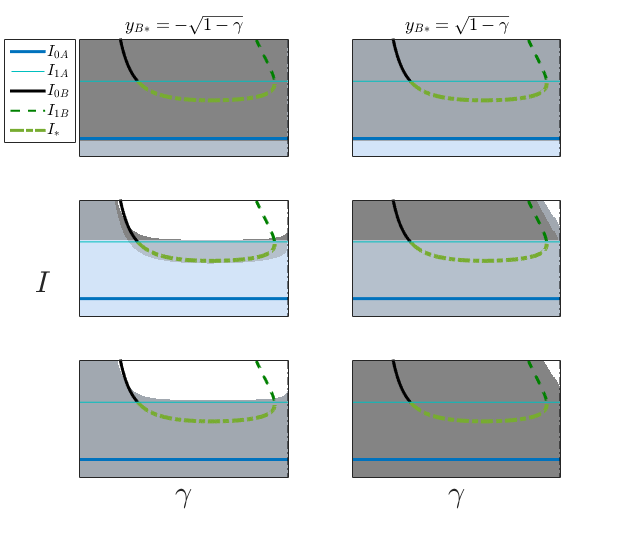}
	\caption{Regions in the $I$-$\gamma$ parameter space distinguishing local stability of the three folded singularities corresponding to $y_{B*} =-\sqrt{1-\gamma}$ (top, middle, and bottom plots on the left) and the three folded singularities corresponding to $y_{B*} =\sqrt{1-\gamma}$ (top, middle, and bottom plots on the right). In white regions, the equilibrium point does not exist. In dark gray regions, both eigenvalues are negative. In light gray regions, the eigenvalues have opposite signs. In light blue regions, both eigenvalues are positive.}\label{evals_folded}
\end{center}
\end{figure*}

\label{classification_sing}

\section{Dynamics by region}
\label{mainres}

In this section, we apply the analytical results for the desingularized system \eqref{desing} to draw conclusions about the original system \eqref{FN2}. In so doing we provide details of the computations used to produce Figure \ref{Igamma_bifn} and the characterization of each of the seven regions, as described in Section~\ref{twoFN}. We prove the stability of limit cycles of the Hopf bifurcations in model neuron $B$. 
 We prove necessary conditions for MMOs and sufficient conditions for phase locking in terms of $I$ and $\gamma$. 

In the following proposition, we compute the value of the unique equilibrium point (\ref{FN2}), when Assumption \ref{assumption} holds, and its stability as a function of $I$ and $\gamma$.

\begin{proposition} \label{original_eq}
  Consider the directed two-FN system (\ref{FN2}) and let Assumption \ref{assumption} hold. For any fixed $I$ and $\gamma$, there exists a unique equilibrium point denoted by $\mathbf{p}_*=(y_{A*}(I), z_{A*}(I), y_{B*}(I,\gamma), z_{B*}(I,\gamma))$. Then,
\begin{enumerate}
\setlength{\itemsep}{0pt}
\setlength{\parskip}{0pt}
\item 
$\mathbf{p}_*$ is nonhyperbolic if $I$ and $\gamma$ satisfy 
\begin{equation}
\begin{aligned}
\sigma_1(I,\gamma) &= 1-b\epsilon-y_{A*}^2 = 0,
\\
\text{or} \qquad 
\sigma_2(I,\gamma) &= 1-b\epsilon-\gamma-y_{B*}^2 = 0.
\end{aligned}
\label{cond_hyp}
\end{equation}
\item 
$\mathbf{p}_*$ is hyperbolic if $I$ and $\gamma$ do not satisfy (\ref{cond_hyp}). If $\sigma_1(I,\gamma)<0$ and $\sigma_2(I,\gamma)<0$, $\mathbf{p}_*$ is attracting. If $\sigma_1(I,\gamma)>0$ and $\sigma_2(I,\gamma)>0$, $\mathbf{p}_*$ is repelling. If $\sigma_1(I,\gamma)\sigma_2(I,\gamma)<0$,  $\mathbf{p}_*$ is a saddle.
\end{enumerate}
\end{proposition}

\begin{proof}

 Solving for the equilibrium point of (\ref{FN2}), we first compute
 $y_{A*}$ as a function of $I$ as
 \begin{align}
 y_{A*} &= 
 { 
 \left(\frac{3(I-a)}{2} + \sqrt{\frac{(3(I-a))^2}{4}+\tilde{b}^3}\right)^{1/3}
} 
+\left(\frac{3(I-a)}{2} - \sqrt{\frac{(3(I-a))^2}{4}+\tilde{b}^3}\right)^{1/3}
\hspace{-1.5em},
\label{yastar}
\end{align}
where $\tilde{b} = \frac{1}{b}-1$. Similarly, by leveraging \eqref{yastar}, we can write $y_{B*}$ as the following function of $I$ and $\gamma$,
{\small
\begin{align}
y_{B*} &= \left(\frac{3(\gamma y_{A*}-a)}{2} + \sqrt{\frac{9(\gamma y_{A*}-a)^2}{4}+\left(\tilde{b}+\gamma\right)^3}\right)^{1/3}  
+\left(\frac{3(\gamma y_{A*}-a)}{2} - \sqrt{\frac{9(\gamma y_{A*}-a)^2}{4}+\left(\tilde{b}+\gamma\right)^3}\right)^{1/3}
\hspace{-1.5em}.	
\label{ybstar}
\end{align}
}
Then $z_{A*}= \frac{1}{b}y_{A*}$ and $z_{B*} = \frac{1}{b}y_{B*}$.

We compute the linearization of (\ref{FN2}) around $\mathbf{p}_*$. We let $\bm{\nu}_A = (y_A, z_A)$, $\bm{\nu}_B = (y_B, z_B)$, and $\bm{\xi} = (\zeta_A, \epsilon\xi_A,\zeta_B,\epsilon\xi_B)$. The Jacobian of \eqref{FN2} evaluated at $p_*$ is
\begin{displaymath}
D_{(\bm{\nu}_A,\bm{\nu}_B)}\bm{\xi}({\mathbf{p}_*}) = \left(\begin{array}{cccc}
1 - y_{A*}^2 & -1 & 0 & 0\\
\epsilon & -b \epsilon & 0 & 0\\
\gamma & 0 & 1-y_{B*}^2-\gamma & -1 \\
0 & 0 & \epsilon & -b \epsilon
\end{array}\right).
\label{Jac}
\end{displaymath}

The linearization is block triangular, so the eigenvalues of the Jacobian are union of the eigenvalues of diagonal blocks. This means that local stability can be determined through linearization of each FN model separately. The eigenvalues for the first and second blocks are
\begin{align*}
\lambda_{1,2} &= \frac{1}{2}\left(1 - y_{A*}^2 -b \epsilon\right)
\pm \frac{1}{2}\sqrt{(1 -y_{A*}^2 -b \epsilon)^2-4 \epsilon (1-b+y_{A*}^2b)},
\\
\lambda_{3,4} &= \frac{1}{2}\left(1 - y_{B*}^2 - \gamma -b \epsilon\right)
\pm \frac{1}{2}\sqrt{(1 -y_{B*}^2 -\gamma -b \epsilon)^2-4 \epsilon (1-b+y_{B*}^2b +\gamma)}.
\end{align*}

The sign of the real part of the eigenvalues will be determined by the sign of the first term. The first term of $\lambda_{1,2}$ is zero when $\sigma_1(I,\gamma) = \lambda_1 + \lambda_2 = 1 - b\epsilon - y_{A*}^2=0$. The first term of $\lambda_{3,4}$ is zero when $\sigma_2(I,\gamma) = \lambda_3 + \lambda_4 = 1-b\epsilon-\gamma - y_{B*}^2=0$.  Thus, $\mathbf{p}_*$ is nonhyperbolic when $\sigma_1 = 0$ or $\sigma_2=0$.

The stability of $\mathbf{p}_*$ when $\sigma_1 \sigma_2 \neq 0$ is derived from the signs of the eigenvalues of $D_{(\bm{\nu}_A,\bm{\nu}_B)}\bm{\xi}$.
\end{proof}
\begin{remark}
The one-dimensional manifolds of nonhyperbolic equilibrium points in $(I,\gamma)$ space, $\{(I,\gamma):\sigma_1(I,\gamma)=0 \text{ or } \sigma_2(I,\gamma)=0 \}$, correspond to the points where $A$ and $B$ undergo Hopf bifurcations.
\end{remark}
%
%
%

\subsection{Quiescence: Region (1)}\label{quiescence}
%
Given the two-FN system (\ref{FN2}) and Assumption \ref{assumption}, if $I<I_{0A}$ then $A$ converges to a single stable equilibrium point, $(y_{A*},z_{A*})$, which is quiescent. The value $y_{A*}$, independent of $\gamma$, is too low to induce firing in $B$, i.e., $B$ is quiescent. The phase plane and nullclines for an $(I,\gamma)$ pair in this regime are shown in Figure \ref{R1}. 

To fully understand the behavior of $B$, we examine the desingularized system \eqref{desing}, which has seven singularities for the parameter values in region (1). These include one stable ordinary singularity, one stable folded singularity, two unstable folded singularities, and three saddle folded singularities (only five of the seven singularities are visible in Figure~\ref{R1}). Due to the presence of the unstable folded singularities in \eqref{desing}, robust families of canards that compose  small oscillations of $B$ could arise in the original system \eqref{FN2} in region (1).
\begin{figure}[t!]
\begin{center}
\includegraphics[width =0.25\textwidth ]{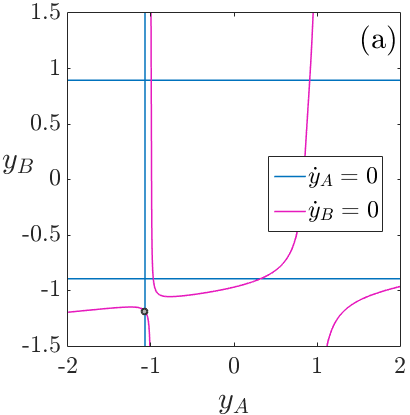}\qquad
\includegraphics[width =0.25\textwidth ]{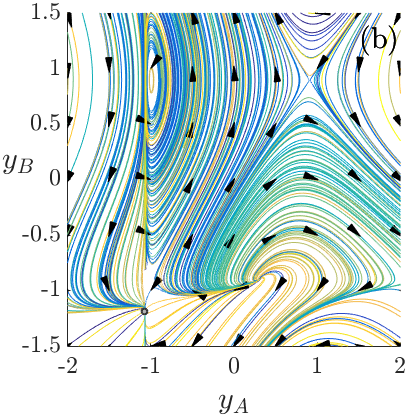}
\caption{
Nullclines (a) and phase plane (b) for the desingularized system \eqref{desing} for parameter values $I = 0.2$ and $\gamma = 0.2$ in region (1), where both $A$ and $B$ are quiescent in the two-FN system \eqref{FN2}. The gray circle indicates the position of the ordinary singularity. Note that $\dot{y}_A$ and $\dot{y}_B$ represent $\frac{dy_A}{d\tau_2}$ and $\frac{dy_B}{d\tau_2}$, respectively.
}
\label{R1}
\end{center}
\end{figure}
%
%
%

\subsection{Hopf bifurcations in $B$: Regions (2), (3), and (4)}
\label{abovebd}
%
For the two-FN system (\ref{FN2}), if $I>I_{1A}$ then $A$ is saturated. We prove conditions for when $B$ will be quiescent, firing, or saturated and provide illustrative examples of the desingularized system nullclines and phase plane for representative parameter values.

\begin{proposition}
Consider the two-FN system (\ref{FN2}) and Assumption \ref{assumption}.	Let $I>I_{1A}$ and $\gamma<1-b\epsilon$. There exist two curves of Hopf bifurcations defined by
\begin{align*}
I_{1B,0B}(\gamma) &=  \tilde{b}y_{A*\pm}+\dfrac{y_{A*\pm}^3}{3}+a,
\end{align*}
where $y_{A*\pm}$ is
{\small
\begin{align*}
y_{A*\pm} = \pm\frac{1}{\gamma}\left(\frac{1}{3} (1-\epsilon b -\gamma)^{3/2}+\left(\tilde{b}+\gamma\right)\sqrt{1-\epsilon b -\gamma}+a\right).
\end{align*}
}$B$ transitions from quiescent to firing through a Hopf bifurcation at $I=I_{0B}(\gamma)$ and from firing to saturated through another Hopf bifurcation at $I=I_{1B}(\gamma)$.

Moreover, there exists $\gamma_*$ such that, for $I<I_{1B}(\gamma_*)$, the following holds. If $\gamma<\gamma*$, the Hopf bifurcation at $I_{0B}(\gamma)$ is supercritical and, if $\gamma>\gamma*$, the Hopf bifurcation at $I_{1B}(\gamma)$ is subcritical.
\end{proposition}
\begin{proof}
The Hopf bifurcations in $B$ occur at nonhyperbolic equilibrium points, which are $y_{B*\pm} = \pm \sqrt{1-\gamma-b\epsilon}$ by Proposition \ref{original_eq}. 

Substituting $y_{B*\pm}=\pm \sqrt{1-\gamma-b\epsilon}$ and \eqref{FN2d} into the equilibrium solution for \eqref{FN2c} gives the critical values
{\small
	\begin{align*}y_{A*\pm} = \pm\frac{1}{\gamma}\left(
	\frac{1}{3} (1-\epsilon b -\gamma)^{3/2}
	+\left(\tilde{b}+\gamma\right)\sqrt{1-\epsilon b -\gamma}+a\right),
	\end{align*}
}Substituting $y_{A*\pm}$ and \eqref{FN2b} into \eqref{FN2a} gives the values 
\begin{align*}
I_{1B,0B} = \tilde{b} y_{A*\pm} +\dfrac{y_{A*\pm}^3}{3}+a.
\end{align*}

For a fixed $\gamma$, we check the conditions of Proposition \ref{Hopfcond} for the bifurcation parameter $I$. First, we transform $(y_{B*},z_{B*})$ to the origin $(0,0)$, by introducing $y_0 = y_B - y_{B*}$ and $z_0 = z_B - z_{B*}$. With this change of variable, the dynamics of $B$ (\ref{FN2c}-\ref{FN2d}) can be expressed as
\begin{equation}
\begin{aligned}
\dfrac{d{y}_0}{dt} 
&= 
(1-\gamma-y_{B*}^2)y_0-\frac{y_0^3}{3}-y_0^2 y_{B*} -z_0,
\\
\dfrac{d{z}_0}{dt} &= \epsilon (y_0-b z_0), 
\end{aligned}
\label{eq2}
\end{equation}
and the Jacobian of \eqref{eq2} evaluated at the origin is 
\begin{displaymath}
J_B(0,0) = \left(\begin{array}{cc}
1-y_{B*}^2-\gamma & -1 \\
\epsilon & -b \epsilon
\end{array}\right).
\end{displaymath}
Now we apply Proposition~\ref{Hopfcond}.

\emph{Condition 1 of Proposition~\ref{Hopfcond}}:
This condition holds because $\Tr(J_B(0,0)) = 0$ at the bifurcation values $I = I_{0B}$ and $I=I_{1B}$.

\emph{Condition 2 of Proposition~\ref{Hopfcond}}:
The second condition holds as well because 
\begin{align*}
\dfrac{\partial}{\partial I}\Re(\lambda_{3,4}(I))\biggr \rvert_{I = I_{0B,1B}} \neq 0,
\end{align*}
where $\lambda_{3,4}$ are given in the proof of Proposition \ref{original_eq}. 

\emph{Condition 2 of Proposition~\ref{Hopfcond}}:
The cubic coefficient $\alpha$ of the Taylor expansion of \eqref{FN2c}--\eqref{FN2d} (Definition \ref{CCdefn}), which determines whether the Hopf bifurcation is subcritical or supercritical \cite{guckenheimer2013nonlinear}, is 
\begin{align}
	\alpha = \frac{1}{8}\left(\dfrac{2b-2b\gamma-b^2\epsilon-1}{1-b^2\epsilon}\right).\notag
\end{align}
At $\gamma=\gamma_*$, $\alpha = 0$ and $B$ undergoes a ``generalized Hopf", or Bautin, bifurcation\cite{takens1973unfoldings, golubitsky1981classification}. For $\gamma>\gamma_*$, $\alpha<0$ and the limit cycles resulting from the Hopf bifurcations are stable (supercritical). Otherwise, the limit cycles are unstable and the bifurcations are subcritical, as for $A$.
\end{proof}
\begin{figure}[b!]
\begin{center}
\includegraphics[width =0.25\textwidth ]{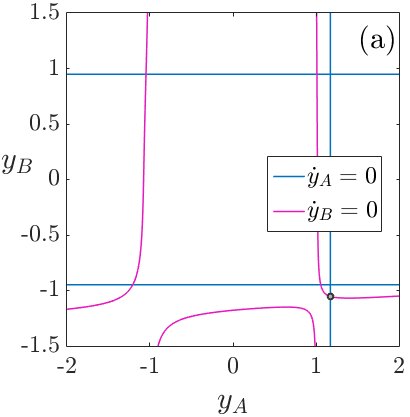}\qquad
\includegraphics[width =0.25\textwidth ]{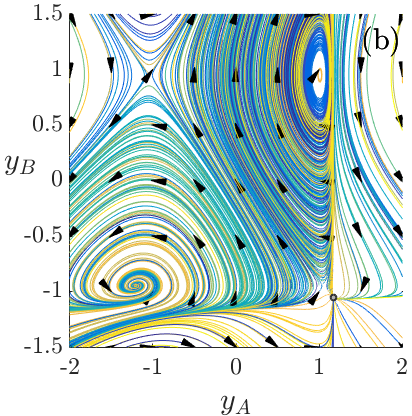}
\caption{
Nullclines (a) and phase plane (b) for the desingularized system \eqref{desing} for parameter values $I = 1.7$ and $\gamma = 0.1$ in region (2), where $A$ is saturated and $B$ is quiescent in the two-FN system \eqref{FN2}.
}\label{R2}
\end{center}
\end{figure}
\begin{remark}
For the two-FN system (\ref{FN2}), given Assumption \ref{assumption}, if $A$ is saturated, then $B$ transitions from quiescent to firing to saturated as a function of $I$ and $\gamma$.
\end{remark}

The phase plane and nullclines for \eqref{desing} in the regions where $A$ is saturated and $B$ is quiescent, firing, and saturated are plotted in Figure \ref{R2}, Figure \ref{R3}, and Figure \ref{R4}, respectively. All three use a fixed value of $I$ given by $I=1.7$. As $\gamma$ is varied, changes to the shape, position, and existence of the nullclines and qualitative differences in the phase portraits can be seen.

In region (2), there are seven singularities, of which only five are visible in Figure~\ref{R2}. The ordinary singularity is stable, and there are three stable folded nodes and three folded saddles. This is a region where $A$ is saturated and $B$ is quiescent, so the two-FN system has a unique stable equilibrium point.
\begin{figure}[t!]
\begin{center}
\includegraphics[width =0.25\textwidth ]{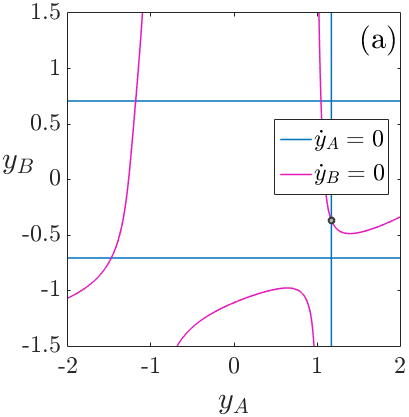}\qquad
\includegraphics[width =0.25\textwidth ]{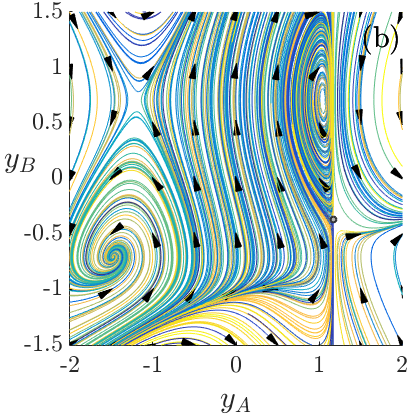}	
\caption{
Nullclines (a) and phase plane (b) for the desingularized system \eqref{desing} for parameter values $I = 1.7$ and $\gamma = 0.5$ in region (3), where $A$ is saturated and $B$ is firing in the two-FN system \eqref{FN2}.
}\label{R3}
\end{center}
\end{figure}
\begin{figure}[b!]
\begin{center}
\includegraphics[width =0.25\textwidth ]{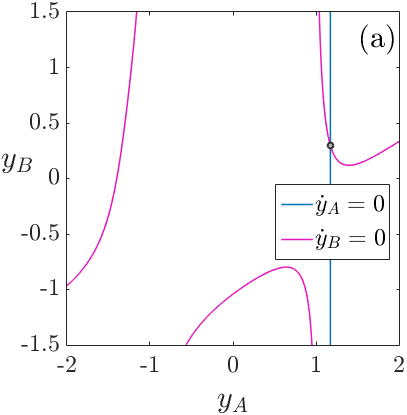}\qquad
\includegraphics[width =0.25\textwidth ]{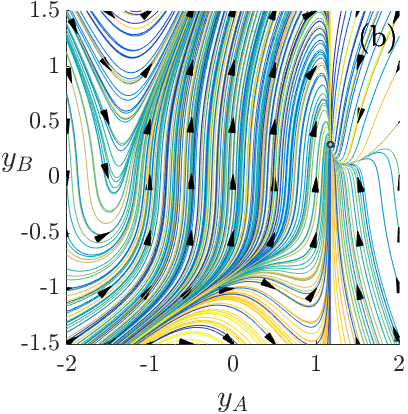}
\caption{
Nullclines (a) and phase plane (b) for the desingularized system \eqref{desing} for parameter values $I = 1.7$ and $\gamma = 1.1$ in region (4), where both $A$ and $B$ are saturated in the two-FN system \eqref{FN2}.
}\label{R4}
\end{center}
\end{figure}

In region (3) there are five singularities, of which only four are visible in Figure \ref{R3}. For low $\gamma$, there are small subregions near the boundary where there are seven singularities. The transition from seven to five singularities occurs through a saddle-node bifurcation between a folded saddle point and a stable folded node. The ordinary singularity is a saddle.

In region (4)  there can be one, three, five, or seven equilibrium points. Parameter choices above $\gamma = 1$, which corresponds to $\gamma = 1-b\epsilon$ in the two-FN system, result in a single singularity as shown in Figure \ref{R4}. Since $\gamma>1$, the folded singularities corresponding to $y_{B*} = \pm\sqrt{1-\gamma}$ no longer exist, and the ordinary singularity is a stable node.
%
%

\subsection{Phase-locking: Region (5)}
%
Now, consider the two-FN system (\ref{FN2}) with $I_{0A}<I<I_{1A}$ and $\gamma>1-b\epsilon$. For all $I$, $\gamma$ in this regime, the linearization of (\ref{FN2}) around the equilibrium point $p_*$ has two eigenvalues $\lambda_{1,2}$ with positive real part, and two eigenvalues $\lambda_{3,4}$ with negative real part. $A$ is in the firing regime since $\Re(\lambda_{1,2})>0$. Furthermore, for all $I$, $\gamma$ under consideration, $\Re(\lambda_{3,4})<0$, so $B$ will follow the limit cycle from $A$ and the oscillators will be phase locked. 

\begin{figure}[t!]
\begin{center}
\includegraphics[width =0.25\textwidth]{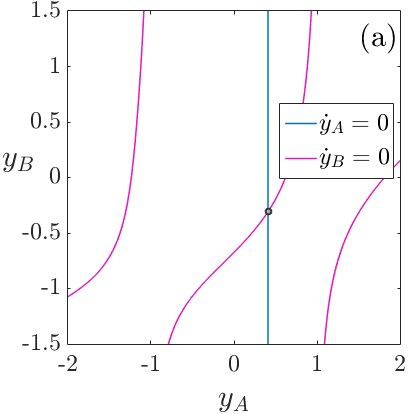}\qquad
\includegraphics[width =0.25\textwidth ]{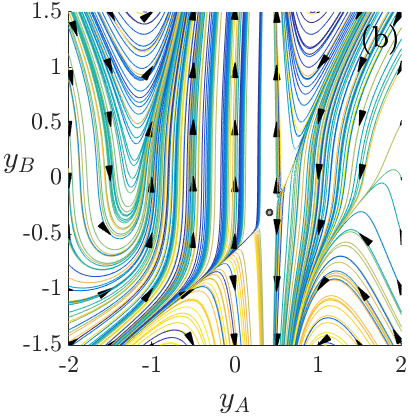}
\caption{
Nullclines (a) and phase plane (b) for the desingularized system \eqref{desing} for parameter values $I = 1.0$ and $\gamma = 1.1$ in region (5), where $A$ is firing and $B$ is phase locked with $A$ in the two-FN system \eqref{FN2}.
}\label{R5}
\end{center}
\end{figure}

\begin{remark}
For the two-FN system (\ref{FN2}), given Assumption \ref{assumption}, if $A$ is firing and $\gamma>1-b\epsilon$, $B$ is firing and $A$ and $B$ are phased locked.
\label{highg}
\end{remark}

The phase plane and nullclines for the two-dimensional, desingularized system \eqref{desing} with $I_{0A}<I<I_{1A}$ and $\gamma>1$ (region (5)) are plotted in Figure \ref{R5}. This is the same region discussed above, where $\gamma>1-b\epsilon$ becomes $\gamma>1$ in the singular limit. In this regime, the singularities at $y_{B*} = \pm\sqrt{1-\gamma}$ no longer exist because $\gamma>1$. There is a single ordinary singularity, which is a saddle. Note that there are still folded singularities corresponding to $y_{A*} = \pm 1$, but that they are the points, $I_{0A}$ and $I_{1A}$, where the stability changes in the ordinary singularity.

%
%
%

\subsection{Mixed mode oscillations: Regions (6) and (7)}
\label{MMObegin}
%
Again consider the two-FN system (\ref{FN2}) when $I_{0A}<I<I_{1A}$ so that $A$ is firing. When $\gamma<1-b\epsilon$, a range of dynamics is observed in simulation. For $\gamma<<1$, the influence of $A$ is small, so $B$ exhibits only small oscillations that stay close to $(y_{B*}, z_{B*})$. As $\gamma$ is increased, the oscillations of $A$ can be large enough to yield a mixed mode oscillation (MMO). As $\gamma$ is increased still further and the influence of $A$ becomes increasingly strong, $B$ approaches 
firing at the same frequency as $A$.

To better understand these transitions, we study the bifurcation diagrams and phase planes of the two-dimensional desingularized system \eqref{desing}.  We prove necessary conditions for the existence of MMOs as a function of $I$ and $\gamma$.

In region (6) there are seven singularities, of which only six are visible in Figure \ref{R6}. In this region, the ordinary singularity of the desingularized system \eqref{desing} is a saddle, whereas the folded singularities include three saddles, one unstable node, and two stable nodes. Co-existence of the unstable folded singularity and the ordinary saddle singularity allows existence of canards and MMOs in the two-FN system.
\begin{figure}[t!]
\begin{center}
\includegraphics[width =0.25\textwidth ]{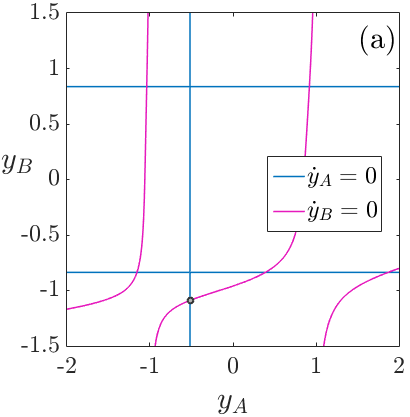}\qquad
\includegraphics[width =0.25\textwidth ]{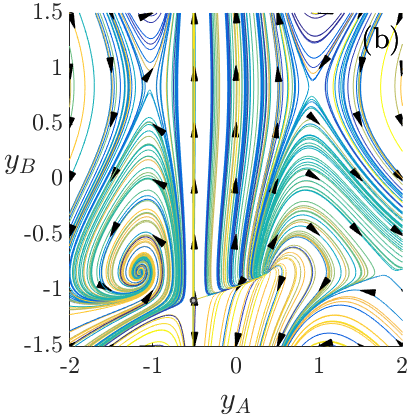}	
\caption{
Nullclines (a) and phase plane (b) for the desingularized system \eqref{desing} for parameter values $I = 0.75$ and $\gamma = 0.3$ in region (6), where $A$ is firing and $B$ may exhibit canards or MMOs in the two-FN system \eqref{FN2}.
}\label{R6}
\end{center}
\end{figure}

In region (7) there are seven singularities, of which only six are visible in Figure \ref{R7}. In this region, the ordinary singularity of the desingularized system \eqref{desing} is an unstable node, whereas the folded singularities include four saddles, and two stable nodes. As a consequence, canards and MMOs do not exist in the two-FN system, and both $A$ and $B$ are firing and phase-locked.
\begin{figure}[t!]
\begin{center}
\includegraphics[width =0.25\textwidth ]{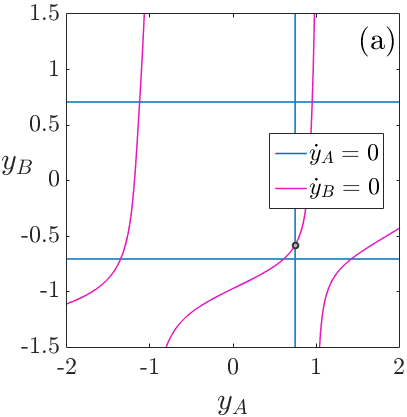}\qquad
\includegraphics[width =0.25\textwidth ]{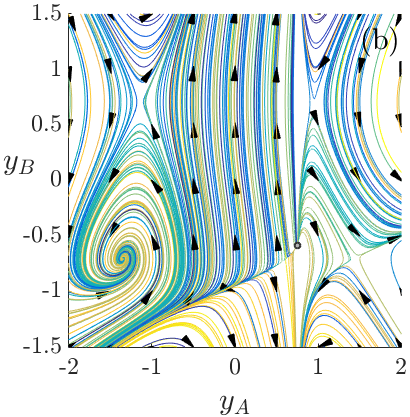}	
\caption{
Nullclines (a) and phase plane (b) for the desingularized system \eqref{desing} for parameter values $I = 1.2$ and $\gamma = 0.5$ in region (7), where $A$ is firing and $B$ phase locked with $A$ in the two-FN system \eqref{FN2}.
}\label{R7}
\end{center}
\end{figure}

We next compute the boundary between regions (6) and (7), shown by $I_{*}$ in Figure~\ref{Igamma_bifn}. The boundary is defined by points at which there is a transcritical bifurcation between the ordinary singularity and a folded singularity, called FSN type II bifurcation, where the ordinary singularity transitions from a saddle to unstable node and the folded singularity transitions from an unstable node to a saddle. This transcritical bifurcation is a known location for the onset of MMOs,  \cite{krupa2014weakly} 
so computing $I_*$ gives necessary conditions for the existence of MMOs. 
\begin{proposition}[Theorem 3.4.1 (modified) \cite{guckenheimer2013nonlinear}]
\label{transcitical:prop:GH}	
A system $\dot{\mathbf{x}} = \mathbf{f}(\mathbf{x},\mu)$, admits a transcritical bifurcation at $(\mathbf{x}_0,\mu_0)$ if
\begin{enumerate}
\setlength{\itemsep}{0pt}
\setlength{\parskip}{0pt}
\item $D_\mathbf{x} \mathbf{f}(\mathbf{x}_0,\mu_0)$ has a simple $0$ eigenvalue with right eigenvector $v$ and left eigenvector $w$.
\item $\bm{w} D^2_{\mathbf{x}\mu} \mathbf{f}(\mathbf{x}_0,\mu_0)\bm{v} \neq 0$.
\item $\bm{w} D^2_\mathbf{x} \mathbf{f}(\mathbf{x}_0,\mu_0)(\bm{v}^\top,\bm{v}^\top)^\top \neq 0$.
\end{enumerate}
\end{proposition}
\begin{proposition}\label{prop:region:6:7}
Consider the desingularized system \eqref{desing} with fixed $\gamma < 1$ and let 
\begin{align}
I_*(\gamma)&= \frac{1}{3b^3\gamma^3}\left(\sqrt{1-\gamma} + \frac{2b\sqrt[3]{1-\gamma}}{3}-ba\right)^3+a.
\label{I_*_GAMMa}
\end{align}      
Then system \eqref{desing} admits a transcritical bifurcation at $(\bm{p}, I_*(\gamma))$, where $\bm{p} = (y_{A*} , y_{B*})^{\top}$ is an ordinary singularity of \eqref{desing}, i.e., $\bm{p}$ solves \eqref{foldedsingb} and \eqref{foldedsingc}.
\end{proposition}
\begin{proof}
To show the transcritical bifurcation, we apply Proposition~\ref{transcitical:prop:GH} to \eqref{desing}. The Jacobian, $D_\mathbf{y} \bm{\rho}$,  of \eqref{desing} is 
\begin{displaymath}
\left(\begin{array}{cc}
(1-y_B^2-\gamma)(1-b+by_A^2) & -2y_B(y_A-bz_A)
\\
-\gamma - 2y_A(y_B-bz_A) & (1-y_{A}^2)(1-b+by_A^2+b\gamma)
\end{array}\right).
\end{displaymath}

\emph{Condition 1 of Proposition~\ref{transcitical:prop:GH}}: 
Evaluating the Jacobian of \eqref{desing} at $(y_A,y_B)^{\top}=\bm{p}$ and $I=I_*$ gives
\begin{displaymath}
D_\mathbf{y} \bm{\rho} (\bm{p},I_*) 
= \left(\begin{array}{cc}
0 & 0
\\
-\gamma& 1-y_{A*}^2
\end{array}\right),
\end{displaymath}
which has a zero eigenvalue with a left eigenvector $\bm{w} = (1,0)$ and a right eigenvector $\bm{v} = \left(1, \frac{\gamma}{1-y_{A*}^2}\right)^{\top}$.

\emph{Condition 2 of Proposition~\ref{transcitical:prop:GH}}:
Taking the derivative of $D_\mathbf{y} \mathbf{\rho}$ with respect to $I$ and evaluating at $(y_A,y_B)^{\top}=\bm{p}$ and $I=I_*$ gives 
\begin{displaymath}
D^2_{\mathbf{y}I} \bm{\rho}(\bm{p},I_*)
= \left(\begin{array}{cc}
0 & 2by_{B*} 
\\
0& 0
\end{array}\right).
\end{displaymath}
Then, multiplying $D^2_{\mathbf{y}\gamma} \bm{\rho}(\bm{p},I_*)$ from left by $\bm{w}$ and from right by $\bm{v}$, we have
\begin{displaymath}
\bm{w} \left(D^2_{\mathbf{y}\gamma}\bm{\rho}(\bm{p},I_*) \right) \bm{v} 
= 
\frac{2b\gamma y_{B*}}{1-y_{A*}^2},
\end{displaymath}
which is always nonzero. 

\emph{Condition 3 of Proposition~\ref{transcitical:prop:GH}}:
Evaluating $D^2_\mathbf{y} \bm{\rho}$ at $\bm{p}=(y_A,y_B)^{\top}$ and $I=I_*$ gives
\begin{displaymath}
D^2_\mathbf{y} \bm{\rho}(\bm{p},I_*) 
= 
\left(\begin{array}{cccc}
0 &2y_{B*}(s-1) & 0 &-2y_{A*} \\
2y_{B*}(s-1)& 0 &-2y_{A*} & 2y_{B*}s
\end{array}\right),
\end{displaymath}
where $s=b(1-y_{A*}^2)$. Then, multiplying $D^2_\mathbf{y} \bm{\rho}$ from left by $\bm{w}$ and from right by $(\bm{v}^\top,\bm{v}^\top)^\top$, we have	
\begin{align}
&\bm{w} \left( D^2_\mathbf{y} \bm{\rho}(\bm{p},I_*)\right) 
\left(\begin{array}{c} 
\bm{v} \\ \bm{v}
\end{array}\right) 
= 
\dfrac{-2\gamma}{1-y_{A*}^2}\left(y_{B*}(1-b(1-y_{A*}^2))+y_{A*}\right),
\label{w-HeSS-v-Quadratic}
\end{align}
which is also nonzero as shown in Figure~\ref{nonzero_second_derivative}. 
\end{proof}
\begin{figure}[h]
\centering
\includegraphics[width =0.4\textwidth]{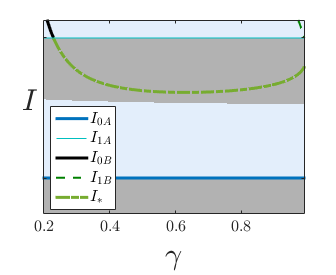}
\caption{Regions in the $I$-$\gamma$ parameter space distinguishing the sign of \eqref{w-HeSS-v-Quadratic}. In the light blue regions, the sign is positive. In the light gray regions, the sign is negative. At the boundaries, the sign becomes zero. For all $(I, \gamma)$ pairs on $I_*$ (shown by the green dashed line), the sign of \eqref{w-HeSS-v-Quadratic} is nonzero, except where $I_*$ intersects $I_{1A}$. The bifurcation at the intersection is a codimension two bifurcation.} \label{nonzero_second_derivative}
\end{figure}

A necessary condition for canard-induced MMOs is the existence of a folded node with a return mechanism, since the small canard solutions that form the small oscillations are only found in this context.\cite{brons2006mixed,krupa2014weakly} The folded node has a corresponding family of canard solutions because there are many trajectories that cross from the stable to the unstable branch of the critical manifold through the folded node singularity. Furthermore, the return mechanism is required for MMOs because, after each relaxation oscillation or canard trajectory, the dynamics must return near the singularity in order for the MMO to persist. In this case, MMOs are only possible for $I<I_*(\gamma)$, since that is where the folded singularity is an unstable folded node and the ordinary singularity is a saddle node.

\begin{remark}
Consider the two-FN system (\ref{FN2}). For $I_{0A}<I<I_*(\gamma)$, this system exhibits MMOs and if $I_*(\gamma)<I<I_{1A}$, it exhibits phase locking.
\end{remark}
\begin{figure*}[t]
\includegraphics[width =0.99\textwidth ]{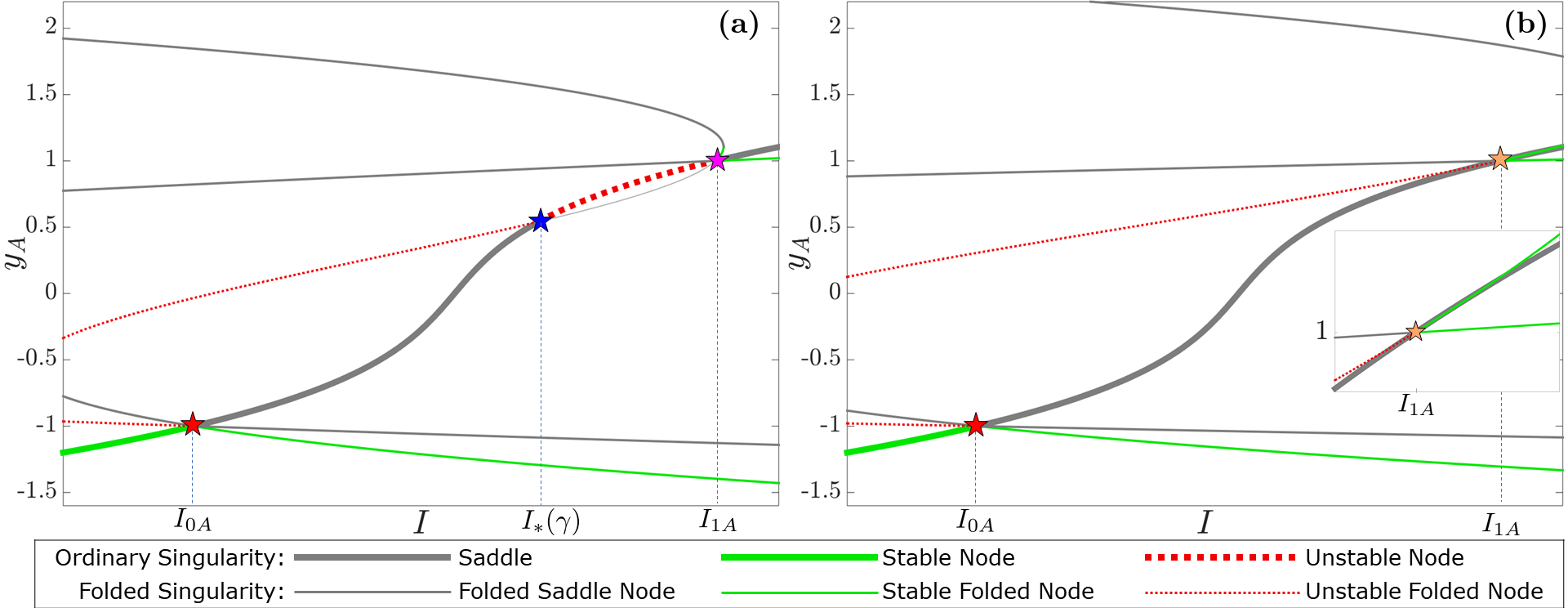}	
\caption{
This figure shows the bifurcation diagrams of the desingularized system \eqref{desing} with bifurcation parameter $I$ for two different values of $\gamma$. (a) When $\gamma=0.4$, a transcritical bifurcation (FSN II) occurs between the ordinary singularity and a folded singularity (blue star). For this choice of $\gamma$, the desingularized system admits another folded singularity (stable node) at $y_{A*}\approx 6$, which is not shown in either figure. (b) When $\gamma = 0.22$, the desingularized system admits a codimension two bifurcation (orange star) where the ordinary singularity remains a saddle, while one folded singularity switches from an unstable folded node to a stable folded node and the other switches from a folded saddle node to a stable folded node.
}
\label{TCd_bifn}
\end{figure*}
     
To highlight the location of the transcritical bifurcation (FSN II) and compare to features in Figure \ref{Igamma_bifn}, we show it as the blue star in the bifurcation diagram of \eqref{desing} in Figure~\ref{TCd_bifn}a.
Here, by treating $I$ as the bifurcation parameter, and maintaining $\gamma$ fixed at value $0.4$, we can observe that the ordinary singularity transitions from stable (thick green) to saddle (thick gray) at $I=I_{0A}$ (red star). Simultaneously, an unstable folded node (thin dashed red) becomes a folded saddle node (thin gray) and a folded saddle node becomes a stable folded node (thin green). For $I>I_{0A}$, MMOs are possible due to the combination of the unstable folded node and unstable ordinary singularity.

Then, at $I=I_*(\gamma)$ (blue star), derived in Proposition~\ref{prop:region:6:7}, the ordinary singularity and a folded singularity swap stability properties in a transcritical bifurcation which can be classified as an FSN II bifurcation. Also, this FSN II bifurcation in the desingularized system \eqref{desing} corresponds to the generalized Hopf bifurcation in the directed two-FN system (\ref{FN2}). For $I>I_*$, \eqref{FN2} exhibits phase locking and MMOs are no longer possible.

At $I=I_{1A}$ (magenta star), the ordinary singularity returns to a saddle and two folded saddle nodes become two stable folded nodes. For $I>I_{1A}$, phase locking in \eqref{FN2} is no longer possible. We also note that when $I$ is just above $I_{1A}$, one folded saddle node merges with a stable folded node through a saddle node bifurcation.

\begin{remark}
A special case of the transcritical (FSN II) bifurcations occurs when $\gamma$ and $I$ satisfy $I_*(\gamma)=I_{1A}$. In this case, there is a codimension two bifurcation where the eigenvalues of the linearization of \eqref{desing} about the ordinary singularity and the eigenvalues of the linearization of \eqref{desing} about the folded singularity, $y_{B*} = -\sqrt{1-\gamma}$, are equal to zero. The codimension two bifurcation is illustrated by the orange star in the bifurcation diagram of \eqref{desing} in Figure~\ref{TCd_bifn}b, for $\gamma = 0.22$. If we fix $I_{0A} < I < I_{1A}$ and decrease the value of $\gamma$ below the codimension two value, then MMOs are always possible.  
\end{remark}

%
\section{Directed Tree of FN model neurons}
\label{chainFN}
%
In this section, we consider an extension of the previous results to a directed chain of coupled FN models. We leverage the connection between the desingularized system and the directed two-FN system to find sufficient conditions for phase locking. 

Consider a system of $k$ FN model neurons with dynamics
\begin{displaymath}
\dot{\mathbf{x}} = \mathbf{f}(\mathbf{x}, \mathbf{I},\bm{\gamma}),
\end{displaymath} 
where $\mathbf{x} \in \mathds{R}^{2k}$, $\mathbf{I} \in \mathds{R}^{k-1}$, and $\bm{\gamma} \in \mathds{R}^{k-1}$. All FN models receive an external input except for the last in the chain. Then, by allowing heterogeneity in the external inputs and coupling strengths, the linearization around the equilibrium point can be expressed as
\begin{displaymath}
D_{\mathbf{x}}\mathbf{f}
= 
\left(\begin{array}{cccccc}
J_1 & 0_{2\times2} & 0_{2\times2} & 0_{2\times2} & \cdots & 0_{2\times2} 
\\
\Gamma_1 & J_2 & 0_{2\times2} & 0_{2\times2} & \cdots & 0_{2\times2}
\\
0_{2\times2} & \Gamma_2 & J_3 & 0_{2\times2} & \ddots & \vdots 
\\
\vdots & \ddots & \ddots &\ddots & \ddots & 0_{2\times2} 
\\
0_{2\times2} & 0_{2\times2} & \ddots &\Gamma_{k-2} & J_{k-1} & 0_{2\times2} 
\\
0_{2\times2} & 0_{2\times2} &\cdots & 0_{2\times2} & \Gamma_{k-1} & J_{k}
\end{array}\right),
\end{displaymath}
where the first diagonal block is given by
\begin{displaymath}
J_1 
= 
\left(\begin{array}{cc}
1-y_{1}^2 & -1
\\
\epsilon & -b\epsilon
\end{array}\right),
\end{displaymath}
and the subsequent diagonal blocks are given by
\begin{displaymath}
J_i 
= 
\left(\begin{array}{cc}
1-y_{i}^2-\gamma_{i-1} & -1
\\
\epsilon & -b\epsilon
\end{array}\right),
\quad i \in \{2,\ldots,k\}.
\end{displaymath}
The blocks on the lower diagonal are
\begin{displaymath}
\Gamma_i 
= 
\gamma_i 
\left(\begin{array}{cc}
0 & 1 \\ 0 & 0
\end{array}\right),
 \quad i \in \{1,\ldots,k-1\}.
\end{displaymath}

Due to the lower block triangular structure of the linearization, local stability of the equilibrium can be determined by studying the eigenvalues of the diagonal blocks. Similar to the analysis at the beginning of Section~\ref{mainres}, we begin by solving for the equilibrium point. The equilibrium of the first model neuron is given by
\begin{align*}
y_{1*} 
&
= \left(\frac{3(I_1-a)}{2} + \sqrt{\frac{(3(I_1-a))^2}{4}+\tilde{b}^3}\right)^{1/3}
+ \left(\frac{3(I_1-a)}{2} - \sqrt{\frac{(3(I_1-a))^2}{4}+\tilde{b}^3}\right)^{1/3}.
\end{align*}
The equilibrium of the $i$-th model neuron is given by
\begin{align*}
y_{i*} &= \left(\frac{3\tilde{I}}{2} + \sqrt{\frac{(3\tilde{I})^2}{4}+\left(\tilde{b}+\gamma_{i-1}\right)^3}\right)^{1/3}
+ \left(\frac{3\tilde{I}}{2} - \sqrt{\frac{(3\tilde{I})^2}{4}+\left(\tilde{b}+\gamma_{i-1}\right)^3}\right)^{1/3},
\end{align*}
where $\tilde{I} = \gamma_{i-1} y_{i-1*}+I_i-a$, $i\in \{2,\ldots,k\}$. The eigenvalues of the individual diagonal blocks are
\begin{align*}
& \lambda_{1,2} 
= 
\frac{1}{2}\left(1 - y_{1*}^2 -b \epsilon\right)
\pm\frac{1}{2}\sqrt{(y_{1*}^2 +b \epsilon-1)^2-4 \epsilon (1-b+y_{1*}^2b)},
\\
& \lambda_{2i-1,2i} 
= 
\frac{1}{2}\left(1 - y_{i*}^2 - \gamma_{i-1} -b \epsilon\right)
\pm\frac{1}{2} \sqrt{(y_{i*}^2 +\gamma_{i-1} +b \epsilon-1)^2-4 \epsilon (1-b+y_{i*}^2b +\gamma_{i-1})},
\end{align*}
where $i =2,\ldots, k$. The Hopf bifurcations in the $i$-th model neuron occur at
\begin{align}
I_{H\pm} 
&
= \pm\frac{1}{3}({1-\gamma_{i-1}-b\epsilon})^{3/2} 
\nonumber 
 \pm\sqrt{1-\gamma_{i-1}-b\epsilon}\left(\tilde{b}+\gamma_{i-1}\right)-\tilde{I}.
\end{align}

As a directed tree can be decomposed into a collection of directed chains, these results generalize to directed trees as well. In Figure \ref{tree}, we illustrate with the directed chain that starts with the light orange FN model and is directed to the right to the cyan FN model. The first FN model (light orange) receives an input $I=1.2$, which ensures that it is in the firing regime. The coupling strength to the second FN model (dark orange) with input $I=0.4$ ensure that the second FN model is in region (6) where MMOs are possible. However, in this case no MMOs are exhibited. The coupling strength to the third FN model (dark cyan) with zero input ensure that it too is in region (6). In this case, MMOs induced by canards are exhibited. The active signal has frequency half that of the first and second FN models. As a result, the input to the fourth FN model (cyan) is an MMO; this case was not covered in our two-FN system analysis. The fourth FN model responds to incoming canards with almost no activity and incoming spikes with a small canard. The frequency of the small canards in the fourth FN model is the same as the frequency of the active signal of the third FN model.
\begin{figure}[h]
\centering
\includegraphics[width =0.5\textwidth]{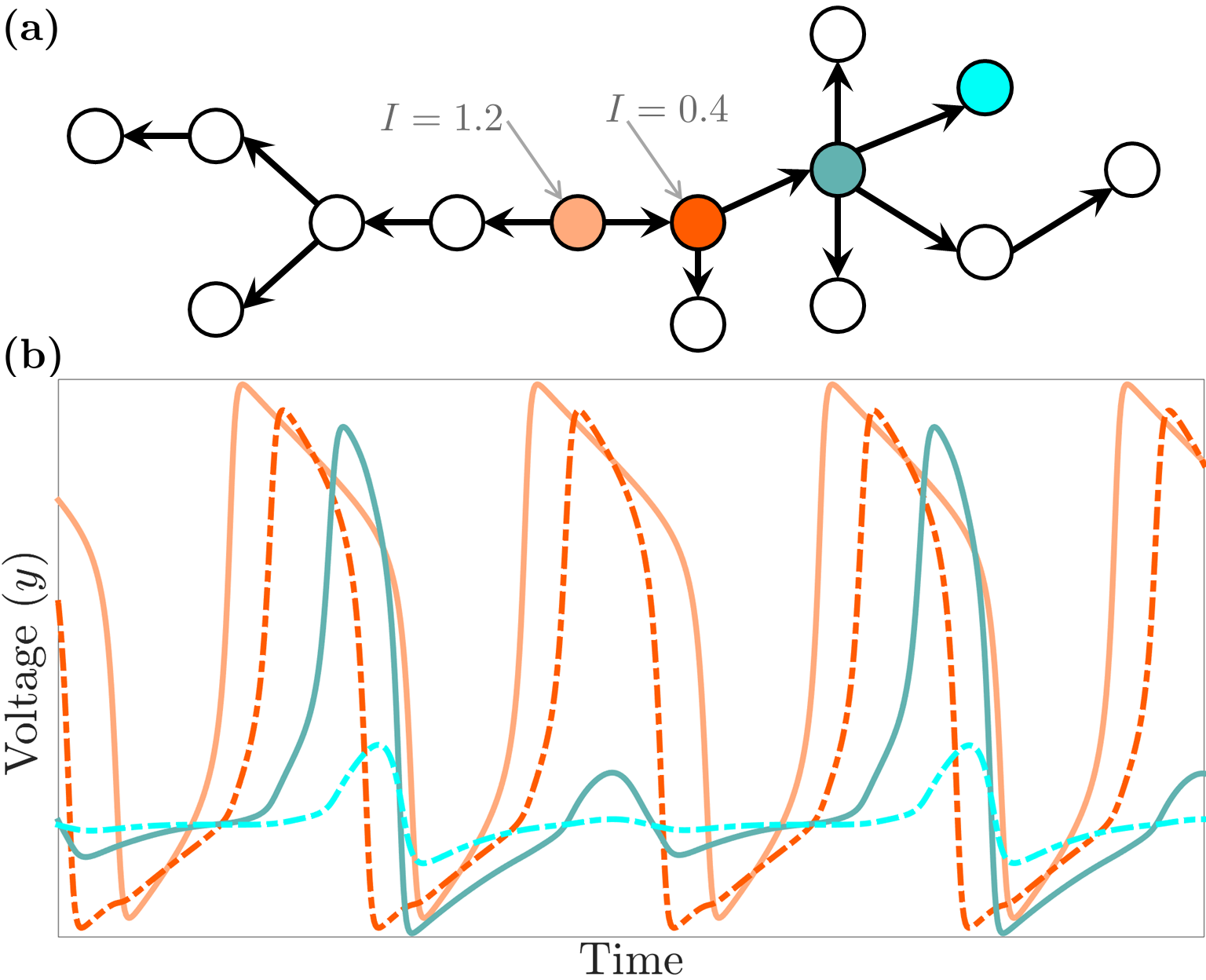}
\caption{
Panel (a) depicts a directed tree graph of FN model neurons with heterogeneous external inputs $I_i$. All edge weights have coupling strength $\gamma = 0.07$. A representative chain is selected and indicated by nodes with colors matching simulation results, which are shown in panel (b). The frequency of the cyan FN models is half of the frequency of the orange FN models.
}
\label{tree}
\end{figure}
%
%
%

%
%
\section{Discussion}
%
In this work, we study a system of two FN model neurons in a setting where the first FN model has a constant external input $I$, the second FN model has no input, and there is a unidirectional coupling with strength $\gamma$ from the first FN model to the second. We study and rigorously characterize all of the different regions of dynamic behavior for the two-FN system in $I$-$\gamma$ space. We prove new necessary conditions in terms of both $I$ and $\gamma$ for the existence of canards and MMOs. We leverage this result to find a similarly new sufficient condition for phase locking and extend to systems of FN models in directed tree networks. We illustrate for a directed chain of four FN models, where canards, MMOs, and frequency halving is observed as predicted.

Further investigation of the two-FN system is needed to determine the threshold between MMOs and canard trajectories without MMOs, which have been observed in simulation. This threshold has been studied numerically, as well as the chaotic behavior at the boundaries between types of MMOs, e.g., in \cite{hoff_numerical_2014}. An analytical understanding of these phenomena would add significantly to the literature on canards and MMOs.

Future directions include consideration of more diverse graph structures that include loops within the graph and a more detailed analysis of the MMOs in these systems. General results have been found for finite dimensional fast-slow systems, which could be applied in this context \cite{wechselberger2012propos}. Incorporating heterogeneous model parameters is another area of future investigation. Changing $\epsilon$ changes the frequency of oscillation and the timescale of the FN model, so a network of FN models with differing values of $\epsilon$ would be a compelling system for exploring canard phenomena in three or more distinct timescales.

%
%
%

%
%
\section{Acknowledgments}
%
This work was jointly supported by the National Science Foundation under NSF-CRCNS grant DMS-1430077 and the Office of Naval Research under ONR grant N00014-14-1-0635. This material is also based upon work supported by the National Science Foundation Graduate Research  Fellowship under grant DGE-1656466. Any opinion, findings, and conclusions or recommendations expressed in this material are those of the authors and do not necessarily reflect the views of the National Science Foundation.
%
%
%
%
%

\end{document}